\documentclass[11pt,a4paper,reqno]{amsart}
\usepackage{amsaddr}
\usepackage{amsmath,amsfonts,amssymb,amsthm,mathrsfs}
\usepackage[utf8]{inputenc}
\usepackage[T1]{fontenc}
\usepackage{lmodern}
\usepackage{latexsym}
\usepackage{cancel}
\usepackage{microtype}
\usepackage{caption}
\usepackage{MnSymbol}
\usepackage{xcolor}
\usepackage{enumerate}
\usepackage[normalem]{ulem}
\usepackage{bbm}
\usepackage{geometry}
\usepackage{marginnote}

\usepackage[colorlinks=true,urlcolor=blue,
linkcolor=blue,citecolor=blue]{hyperref}
\usepackage{mffbsde-style}

\usepackage{graphics,tikz}
\usepackage{caption}

\newcommand{\sgn}{\mathrm{sgn}}

\def\XXint#1#2#3{{\setbox0=\hbox{$#1{#2#3}{\int}$ }
\vcenter{\hbox{$#2#3$ }}\kern-.6\wd0}}

\newcommand{\essinf}{\mathop{\mathrm{ess\,inf}}}
\newcommand{\esssup}{\mathop{\mathrm{ess\,sup}}}

\newtheorem{theorem}{\bf Theorem}[section]
\newtheorem{proposition}[theorem]{\bf Proposition}

\theoremstyle{definition}
\newtheorem{definition}[theorem]{Definition}

\newtheorem{remark}[theorem]{Remark}

\numberwithin{equation}{section}

\begin{document}

\papertitle[MF doubly reflected FBSDEs]{Mean-Field Doubly Reflected Forward--Backward SDEs with Optional Barriers and $L^p$-Data}
\paperauthors{%
  \AuthorBlock{Erhan Bayraktar\textsuperscript{*}}{Department of Mathematics\\ University of Michigan}{erhan@umich.edu}
  \hfill
  \AuthorBlock{Maurycy Rzymowski}{Faculty of Mathematics and Computer Science\\ Nicolaus Copernicus University}{maurycyrzymowski@mat.umk.pl}
}
\paperabstract{We study mean-field doubly reflected forward-backward stochastic differential equations with two optional barriers satisfying a strong Mokobodzki condition. For $L^p$-data, $p\in(1,2]$, we prove existence and uniqueness on sufficiently short time horizons when the coefficients may depend on the joint law of $(X,Y,Z)$. Under an additional monotonicity condition and using an exponentially weighted norm, we also obtain a global-in-time result for $p=2$. The setting is motivated by recursive mean-field Dynkin games and game-option valuation with irregular payoff barriers.}
\paperkeywords{Mean-field FBSDEs, doubly reflected BSDEs, optional barriers, Dynkin games, Mokobodzki condition}
\papermsc{Primary 60H10; Secondary 60H30, 91A15, 91G80}

\makepaperfrontmatter
\footnotetext[1]{Erhan Bayraktar is supported in part by the National Science Foundation Grant in Applied Mathematics \href{https://www.nsf.gov/awardsearch/showAward?AWD_ID=2507940\&HistoricalAwards=false}{DMS-2507940}.}

\section{Introduction}

Let $W$ be a standard $d$-dimensional Brownian motion on a probability space $(\Omega,\mathcal{F},P)$, let $T>0$, and let $\mathbb{F}:=(\mathcal{F}_t)_{t\in[0,T]}$ be the augmented filtration generated by $W$. In this paper, we study mean-field forward-backward stochastic differential equations with reflection (MF-FBSDERs for short) of the following form:
\begin{equation}\label{9paz1}
\left\{
\begin{alignedat}{2}
X_t&=\xi+\int_0^t B(r,X_r,Y_r,Z_r,\mathcal{L}(X_r,Y_r,Z_r))\,dr+\sigma W_t,
    && t\in[0,T],\\
Y_t&=h(X_T,\mathcal{L}(X_T))
    +\int_t^T f(r,X_r,Y_r,Z_r,\mathcal{L}(X_r,Y_r,Z_r))\,dr\\
&\quad+R_T-R_t-\int_t^T Z_r\,dW_r,
    && t\in[0,T],\\
L_t&\le Y_t\le U_t,
    && t\in[0,T],\\
0&=\int_0^T\big(Y_{r-}-\limsup_{s\uparrow r}L_s\big)\,dR^{*,+}_r\\
&\quad+\int_0^T\big(\liminf_{s\uparrow r}U_s-Y_{r-}\big)\,dR^{*,-}_r,\\
0&=\sum_{0\le r<T}(Y_r-L_r)\max\{R_{r+}-R_r,0\}\\
&\quad+\sum_{0\le r<T}(U_r-Y_r)\max\{R_r-R_{r+},0\}.
\end{alignedat}
\right.
\end{equation}
Here $\xi$ is an $\mathcal{F}_0$-measurable random variable, $\mathcal{L}(X_r,Y_r,Z_r)$ denotes the joint law of $X_r$, $Y_r$, and $Z_r$, and
\[
B,f:[0,T]\times\mathbb{R}\times\mathbb{R}\times\mathbb{R}^d\times\mathcal{P}_1(\mathbb{R}^{d+2})\to\mathbb{R}
\]
are the forward and backward generators. For $k\geq 1$ and $m\geq 1$, $\mathcal{P}_k(\mathbb{R}^m)$ denotes the set of probability measures on $(\mathbb{R}^m,\mathcal{B}(\mathbb{R}^m))$ with finite $k$th moment. Moreover, $h:\mathbb{R}\times\mathcal{P}_1(\mathbb{R})\to\mathbb{R}$ is the terminal function, and $L, U$ are $\mathbb{F}$-optional processes of class (D). We write $R^*$ for the càdlàg part of $R$ and $R^{*,+}, R^{*,-}$ for its Jordan decomposition. Since $Z$ is defined only up to $dt\otimes dP$-null sets, we fix throughout the paper a progressively measurable representative of $Z$. All occurrences of $\mathcal{L}(X_t,Y_t,Z_t)$ and of Wasserstein distances involving $Z_t$ are understood relative to this representative, and identities in which they appear under time integration are unaffected by changes on $dt\otimes dP$-null sets.

The first results on forward-backward stochastic differential equations (FBSDEs for short), that is, on \eqref{9paz1} without barriers and without mean-field terms, were obtained by Antonelli \cite{Antonelli}, a few years after the pioneering work \cite{PP} on backward stochastic differential equations. The theory has since become a central tool in stochastic analysis and in applications, especially in economics and finance; see, for instance, \cite{Antonelli,DE,EPQ}.

An FBSDE consists of two coupled stochastic differential equations: a forward equation describing the evolution of a state process, and a backward equation that is often related to optimization or valuation problems, such as the pricing of contingent claims.

Since \cite{Antonelli}, the theory of FBSDEs has developed in many directions, both through probabilistic methods and through connections with PDEs. We refer to \cite{Ma} for broader background.

Two developments are particularly relevant here. First, El Karoui et al. \cite{EKPPQ} introduced reflected backward stochastic differential equations with one barrier, and Cvitanic and Karatzas \cite{CK} developed the two-barrier case. In these works, the authors considered Brownian filtrations and square-integrable continuous barriers $L,U$ satisfying the Mokobodzki condition. The corresponding RBSDEs have the form
\begin{equation}
\label{9paz2}
\left\{
\begin{alignedat}{2}
Y_t&=\xi+\int_t^T f(r,Y_r,Z_r)\,dr+\int_t^T dR_r-\int_t^T Z_r\,dW_r,
    && t\in[0,T],\\
L_t&\leq Y_t\leq U_t,
    && t\in[0,T],\\
0&=\int_0^T(Y_r-L_r)\,dR^+_r
  =\int_0^T(U_r-Y_r)\,dR^-_r.
\end{alignedat}
\right.
\end{equation}
where the solution $(Y,Z,R)$ has continuous $Y$ and a finite-variation process $R$ starting from zero. These developments also led to the study of FBSDEs with reflection (FBSDERs), in which a forward equation is coupled with a reflected backward equation such as \eqref{9paz2}. A reflected FBSDE (respectively, RBSDE) extends a classical FBSDE (respectively, BSDE) by adding a reflection mechanism that keeps the solution inside prescribed constraints. This mechanism is encoded by the finite-variation reflection process $R$ (or by the reflection measure $dR$), which prevents the relevant component, typically $Y$ in the backward equation, from crossing the prescribed lower or upper barrier; see \cite[Chapter 7]{Ma}.

Second, Carmona and Delarue \cite{CD} introduced mean-field FBSDEs (MF-FBSDEs), which are equations of the form
\begin{equation*}
\left\{
\begin{alignedat}{2}
X_t&=\xi+\int_0^t B(r,X_r,Y_r,Z_r,\mathcal{L}(X_r,Y_r,Z_r))\,dr+\sigma W_t,
    && t\in[0,T],\\
Y_t&=h(X_T,\mathcal{L}(X_T))
    +\int_t^T f(r,X_r,Y_r,Z_r,\mathcal{L}(X_r,Y_r,Z_r))\,dr\\
&\quad-\int_t^T Z_r\,dW_r,
    && t\in[0,T].
\end{alignedat}
\right.
\end{equation*}
This has become a rapidly developing area of stochastic differential equations, with a wide range of applications in optimization, control, and game theory; see \cite{CD2} for a comprehensive account and \cite{Yong} for related $L^p$-theory for FBSDEs.

MF-FBSDEs generalize FBSDEs by incorporating the mean-field effect, namely the dependence of a representative agent on the distribution of the population. Thus the trajectory and value of a single agent depend not only on its own state but also on aggregate statistical features of the population. Such models arise in stochastic game theory and in large interacting particle systems.

The goal of this work is to combine these two directions:
\[
\mathrm{FBSDER}\quad+\quad\mathrm{MF-FBSDE}\quad=\quad\mathrm{MF-FBSDER},
\]
This synthesis is represented by \eqref{9paz1}. We consider a forward-backward problem with a mean-field component and simultaneous lower and upper constraints on the backward component $Y$ through the barriers $L$ and $U$.

One of the closest works to our setting is \cite{LMXZ}, whose motivation is to combine reflected FBSDEs with mean-field coefficients depending on the law of the solution. This builds on mean-field BSDEs and related PDEs \cite{BDL,BLP}, reflected mean-field BSDEs \cite{Li2014,LL2012}, and reflected FBSDEs \cite{HLW,LPL}.

Very recently, Cosso, D'Andolfi, and Dumitrescu \cite{CDD2026} introduced a coupled reflected forward--backward McKean--Vlasov system for optimal-stopping mean-field games with randomized stopping strategies. Their square-integrable formulation involves a single obstacle and an endogenous survival process representing the randomized stopping strategy, and it yields equilibrium, finite-player approximation, and PDE results. In their system, the forward and backward components are coupled through the survival process; conditional on that process, the forward dynamics do not depend directly on $(Y,Z)$. This differs from the fully coupled system considered here, whose coefficients may depend directly on $(X,Y,Z)$ and their joint law.

The present contribution is complementary and should be understood as a synthesis of several strands that are usually treated separately: mean-field dependence, forward-backward coupling, two reflecting barriers, optional-barrier regularity, $L^p$ data with $p\in(1,2]$, and law dependence involving $Z$. Table \ref{tab:comparison} gives a schematic comparison with representative nearby settings. In the table, MF means mean-field dependence, F--B means forward-backward coupling, Two means two reflecting barriers, Optional means optional barriers, and Law in $Z$ records whether the coefficients may depend on the law of the martingale integrand.

\begin{center}
\refstepcounter{table}
\label{tab:comparison}
{\small\textbf{Table \thetable.} Schematic comparison with nearby classes of equations.\par}
\medskip
\scriptsize
\begin{tabular}{lcccccc}
\hline
Setting & MF & F--B & Two & Optional & Data & Law in $Z$\\
\hline
Double RBSDEs \cite{CK} & no & no & yes & no & $L^2$ & no\\
Optional RBSDEs \cite{GIOQ,KRS,KRz1} & no & no & yes & yes & $L^p$ & no\\
Mean-field FBSDEs \cite{CD,CD2,BZ} & yes & yes & no & no & varies & sometimes\\
MF reflected FBSDEs \cite{LMXZ} & yes & yes & no & no & $L^2$ & limited\\
Randomized OS--MFGs \cite{CDD2026} & yes & via $L$ & no & no & $L^2$ & no\\
This paper & yes & yes & yes & yes & $p\in(1,2]$ & yes\\
\hline
\end{tabular}
\end{center}

The optional-barrier feature is particularly important. The theory of RBSDEs with optional barriers has developed only recently, and relatively few papers are available in this area; see, for example, \cite{KRz1,KRS,GIOQ}. Optional-barrier RBSDEs are well suited to problems in finance, insurance, and stochastic control in which payoff processes may jump or fail to be right-continuous.

We prove existence and uniqueness for \eqref{9paz1} by combining results for the forward and reflected backward equations and by analyzing the fixed-point structure that couples them. In the first part of the paper, an important limitation is the length of the time horizon. Specifically, we assume that
\begin{equation}\label{9paz4}
T<C
\end{equation} 
for a sufficiently small constant $C>0$. In our framework, it is difficult to obtain existence for \eqref{9paz1} on arbitrary time intervals using only stochastic analysis. For classical mean-field FBSDEs, longer-time results can sometimes be obtained through PDE methods and by gluing local solutions; see \cite{CD2}. In the present setting, however, the PDE theory corresponding to reflected FBSDEs, and in particular to optional-barrier RBSDEs, is not yet available.

In the square-integrable case $p=2$, our second method bypasses \eqref{9paz4}: under additional assumptions on $f$, and by working with an exponentially weighted norm, we obtain a solution on an arbitrary interval $[0,T]$. A similar method for MF-FBSDEs was used in \cite{BZ} for an infinite time horizon.

Natural continuations of the present work include $N$-player approximations and $\varepsilon$-Nash equilibria for the corresponding finite-player Dynkin games, Markovian double-obstacle variational inequalities of McKean--Vlasov type, particle or deep-BSDE numerical schemes, weaker Mokobodzki assumptions, and common-noise or jump extensions. We do not pursue these questions here.

\subsection{A mean-field Dynkin game interpretation}

A central feature of \eqref{9paz1} is that the reflected component is the backward variable $Y$, rather than the forward state variable $X$. Thus the reflection acts on the continuation value of a representative agent. This is natural in recursive Dynkin games and game-option valuation: the lower obstacle represents the payoff available to one player by exercising, stopping, exiting, or claiming protection, while the upper obstacle represents the payoff imposed by the counterparty through cancellation, calling, intervention, or termination. The finite-variation term $R=R^+-R^-$ is the minimal correction that keeps the continuation value inside the admissible payoff band $L\le Y\le U$.

Doubly reflected BSDEs are naturally connected with Dynkin games; see, for example, \cite{CK,GIOQ,KRz1}. In the present mean-field setting, the value process of a representative agent is constrained by stopping and cancellation payoffs, while the dynamics and payoff rates depend on the population law. This viewpoint is close in spirit to mean-field optimal stopping problems, where rewards depend on the distribution of the population; see \cite{Bertucci,BDT}. To make the interpretation explicit, fix a candidate flow of laws $\mu=(\mu_t)_{t\in[0,T]}$ with $\mu_t\in\mathcal{P}_p(\mathbb{R}^{d+2})$, write $\mu^X_t$ for its first marginal, and let $\mathcal{T}_t$ denote the stopping times taking values in $[t,T]$. For fixed $\mu$, set
\[
g^\mu_s(y,z):=f(s,X^\mu_s,y,z,\mu_s).
\]
For $\tau,\theta\in\mathcal{T}_t$, define
\[
\begin{aligned}
\xi_{\tau,\theta}^\mu
&=L_\tau\mathbf{1}_{\{\tau\le \theta,\ \tau<T\}}
  +U_\theta\mathbf{1}_{\{\theta<\tau\}}\\
&\quad
  +h(X_T^\mu,\mu_T^X)\mathbf{1}_{\{\tau=\theta=T\}}.
\end{aligned}
\]
The fixed-population recursive Dynkin game is then formulated through the $g^\mu$-evaluation
\[
J_t^\mu(\tau,\theta)
:=\mathcal{E}^{g^\mu}_{t,\tau\wedge\theta}
  \left[\xi_{\tau,\theta}^\mu\right],
\]
where $\mathcal{E}^{g^\mu}$ is the nonlinear evaluation defined by the BSDE with generator $g^\mu$. Its value is, schematically,
\[
Y_t^\mu
=\esssup_{\tau\in\mathcal{T}_t}
  \essinf_{\theta\in\mathcal{T}_t}
  J_t^\mu(\tau,\theta).
\]
Here $Y^\mu$ is the continuation value. The process $Z^\mu$ is the martingale integrand produced by the associated BSDE/DRBSDE, not an exogenous component of the stopping payoff; equivalently, it is the predictable process multiplying $dW_t$ in the stochastic integral. In financial terminology, $Z^\mu$ may be interpreted as the hedging or risk-exposure component associated with the value process.
The mean-field fixed point then imposes the equilibrium consistency condition
\[
\mu_t=\mathcal{L}(X^\mu_t,Y^\mu_t,Z^\mu_t),\qquad t\in[0,T].
\]
Thus the reflected BSDE gives the Dynkin-game value for a fixed population law, and the mean-field fixed point makes the representative agent consistent with the population distribution. This should be understood as a verification interpretation rather than a separate game theorem proved in the present paper. A precise mean-field game statement can be obtained under additional assumptions ensuring that, for each admissible flow $\mu$, the fixed-population recursive Dynkin game admits a value represented by the doubly reflected BSDE with generator $g^\mu$, barriers $L,U$, and terminal payoff $h(X_T^\mu,\mu_T^X)$, together with suitable saddle points or $\varepsilon$-saddle points. Under these verification hypotheses, a solution of the mean-field system \eqref{9paz1} satisfying the consistency condition above yields a representative-agent mean-field equilibrium. Proving such a verification theorem, and especially an $N$-player approximation or $\varepsilon$-Nash result, is beyond the scope of this paper. For equilibrium and finite-player approximation results in the distinct setting of one-sided reflected optimal-stopping mean-field games with randomized strategies, see \cite{CDD2026}.

This gives a concrete interpretation of mean-field game options and cancellable contracts, in the tradition of game options and Israeli options \cite{KiferGameOptions}. Here $Y$ is the contract value, $L$ is the holder's exercise, conversion, or surrender payoff, and $U$ is the issuer's cancellation payoff. The law dependence captures aggregate market effects such as liquidity, crowding, systemic volatility, or the distribution of agents' wealth and positions. Optional barriers are useful because exercise and cancellation payoffs may jump at announcement times, default times, coupon dates, or other irregular information times. The same continuation-value interpretation also covers surrenderable insurance and annuity contracts, where $Y$ is a reserve or policy value; mean-field entry--exit or investment games, where $Y$ is the value of remaining active; and systemic-risk or capital-reserve models, where $Y$ represents a solvency-adjusted value, reserve, or continuation utility constrained by regulatory or contractual limits.

\begin{remark}
The theorems below treat $L$ and $U$ as given optional processes satisfying terminal compatibility and a strong Mokobodzki condition. This already covers applications in which the state dynamics, running payoff, and terminal payoff are distribution-dependent while the admissible value band is specified exogenously. A stronger model would allow the barriers themselves to depend on the population, for instance
\[
L^\mu_t=\ell(t,X_t,\mu_t),\qquad U^\mu_t=u(t,X_t,\mu_t),
\]
where $\mu_t=\mathcal{L}(X_t,Y_t,Z_t)$. Such barriers correspond to exercise and cancellation payoffs that depend on the population distribution, and they are especially natural in large markets with liquidity effects, surrenderable insurance with population-dependent lapse behavior, and systemic-risk models with aggregate solvency constraints. Proving this extension would require stability estimates for doubly reflected BSDEs with respect to the barriers and is therefore beyond the present results.
\end{remark}

\subsection{Notation and preliminaries}

For $x\in\mathbb{R}^d$, $|x|$ denotes the Euclidean norm. Let $\mathcal{S}_{\mathbb{F}}(0,T)$ be the set of all $\mathbb{F}$-progressively measurable, $\mathbb{R}$-valued processes on $[0,T]$. We denote by $\mathcal{S}^p_{\mathbb{F}}(0,T)$ the set of all $Y\in\mathcal{S}_{\mathbb{F}}(0,T)$ such that
\[
||Y||_{\mathcal{S}^p_{\mathbb{F}}(0,T)}:=\Big(\mathbb{E}\sup_{0\le t\le T}|Y_t|^p\Big)^{\frac{1}{p}}<\infty.
\]
Let $q\ge 1$. We denote by $L^{p,q}_{\mathbb{F}}(\alpha,\beta)$ the set of all $\mathbb{F}$-progressively measurable, $\mathbb{R}$-valued processes $X=(X_t)_{t\in[0,T]}$ such that
\[
||X||_{L^{p,q}_{\mathbb{F}}(\alpha,\beta)}:=\Bigg(\mathbb{E}\Big(\int^{\beta}_{\alpha}|X_r|^p\,dr\Big)^{\frac{q}{p}}\Bigg)^{\frac{1}{q}}<\infty.
\]
$L^p_{\mathbb{F}}(\alpha,\beta)$ denotes $L^{p,p}_{\mathbb{F}}(\alpha,\beta)$.

Let $\mathcal{G}\subset\mathcal{F}$. We write $L^p(\mathcal{G})$ for the set of all $\mathcal{G}$-measurable random variables $X$ such that
\[
||X||_{L^p}:=\Big(\mathbb{E}|X|^p\Big)^{\frac{1}{p}}<\infty.
\]
We denote by $\mathcal{H}_{\mathbb{F}}(\alpha,\beta)$ the space of all $\mathbb{F}$-progressively measurable, $\mathbb{R}^d$-valued processes $Z=(Z_t)_{t\in[0,T]}$ such that
\[
\int^{\beta}_{\alpha}|Z_r|^2\,dr<\infty\quad P\mbox{-a.s.}
\]
$\mathcal{H}^s_{\mathbb{F}}(\alpha,\beta)$, $s>0$, is the subspace of $Z\in\mathcal{H}_{\mathbb{F}}(\alpha,\beta)$ satisfying
\[
\mathbb{E}\Big(\int^{\beta}_{\alpha}|Z_r|^2\,dr\Big)^{\frac{s}{2}}<\infty.
\] 
We denote by $\mathcal{V}_{\mathbb{F}}(\alpha,\beta)$ (resp. $\mathcal{V}^+_{\mathbb{F}}(\alpha,\beta)$) the space of $\mathbb{F}$-progressively measurable, $\mathbb{R}$-valued processes $V=(V_t)_{t\in[0,T]}$ with finite variation (resp. nondecreasing paths) on $[[\alpha,\beta]]$. The spaces $\mathcal{V}_{0,\mathbb{F}}(\alpha,\beta)$ and $\mathcal{V}^+_{0,\mathbb{F}}(\alpha,\beta)$ consist of those processes in $\mathcal{V}_{\mathbb{F}}(\alpha,\beta)$ and $\mathcal{V}^+_{\mathbb{F}}(\alpha,\beta)$, respectively, satisfying $V_{\alpha}=0$. Finally, $\mathcal{V}^p_{\mathbb{F}}(\alpha,\beta)$ and $\mathcal{V}^{+,p}_{\mathbb{F}}(\alpha,\beta)$ denote the corresponding subspaces for which $E|V|^p_{\alpha,\beta}<\infty$, where $|V|_{\alpha,\beta}$ is the total variation of $V$ on $[[\alpha,\beta]]$.

Let $V\in\mathcal{V}_{\mathbb{F}}(0,T)$. By $V^*$ we denote the c\`adl\`ag part of the process $V$, i.e.
\[
V^*_t=V_t-\sum_{0\le r<t}\Delta^+V_r.
\]

Throughout the paper, relations between random variables are understood to hold $P$-a.s. For processes $X^1=(X^1_t)_{t\in[0,T]}$ and $X^2=(X^2_t)_{t\in[0,T]}$, we write $X^1\le X^2$ if $X^1_t\le X^2_t$ for all $t\in[0,T]$, $P$-a.s.

For an $\mathbb{F}$-optional process $X=(X_t)_{t\in[0,T]}$, we set $\overrightarrow{X}_s=\limsup_{r\uparrow s}X_r$ and $\underrightarrow{X}_s=\liminf_{r\uparrow s}X_r$ for $s\in(0,T]$. We also set $\overleftarrow{X}_s=\limsup_{r\downarrow s}X_r$ and $\underleftarrow{X}_s=\liminf_{r\downarrow s}X_r$ for $s\in[0,T)$.

We consider the set $\mathcal{P}_p(\mathbb{R}^{d+2})$, $p\ge 1$, endowed with the $p$-Wasserstein metric: for $\eta,\eta'\in\mathcal{P}_p(\mathbb{R}^{d+2})$,
\[
\begin{aligned}
\mathcal{W}_p(\eta,\eta')
:=\inf\Big\{&
\Big(\int_{\mathbb{R}^{d+2}\times\mathbb{R}^{d+2}}|x-y|^p\gamma(dx,dy)\Big)^{\frac{1}{p}}:\,
\gamma\in\mathcal{P}_p(\mathbb{R}^{2d+4}),\\
&\gamma(\cdot\times\mathbb{R}^{d+2})=\eta,\quad
\gamma(\mathbb{R}^{d+2}\times\cdot)=\eta'\Big\}.
\end{aligned}
\]
We use the same convention for $\mathcal{P}_p(\mathbb{R})$.

\section{Mean-field Stochastic Differential Equations}

Let $p\in(1,2]$, $Y\in\mathcal{S}^p_{\mathbb{F}}(0,T)$, and $Z\in\mathcal{H}^p_{\mathbb{F}}(0,T)$.

\begin{definition}
We say that a stochastic process $X$ is a solution to the mean-field stochastic differential equation on $[0,T]$ with initial condition $\xi$, generator $B$, diffusion coefficient $\sigma$, and given processes $Y,Z$ (\textnormal{MF-SDE}$^T(\xi,B,\sigma,Y,Z)$ for short) if
\begin{enumerate}
\item[(a)] $\int^T_0|B(r,X_r,Y_r,Z_r,\mathcal{L}(X_r,Y_r,Z_r))|\,dr<\infty$,
\item[(b)] $X_t=\xi+\int^t_0 B(r,X_r,Y_r,Z_r,\mathcal{L}(X_r,Y_r,Z_r))\,dr+\sigma W_t$, $t\in[0,T]$.
\end{enumerate}
\end{definition}

\noindent We impose the following assumptions.

\begin{enumerate}
\item[(A1)] There exist $L_x,L_y,L_z,\gamma\ge0$ such that
\[
|B(t,x,y,z,\eta)-B(t,x',y',z',\eta')|\le L_x|x-x'|+L_y|y-y'|+L_z|z-z'|+\gamma\mathcal{W}_{p}(\eta,\eta')
\]
for $t\in[0,T]$, $x,x',y,y'\in\mathbb{R}$, $z,z'\in\mathbb{R}^d$, $\eta,\eta'\in\mathcal{P}_{p}(\mathbb{R}^{d+2})$,
\item[(A2)] $\xi\in L^p(\mathcal{F}_0)$, $B(\cdot,0,0,0,\mathcal{L}(0,0,0))\in L^{1,p}_{\mathbb F}(0,T)$.
\item[(A3)] $\int^T_0 |B(r,x,0,0,\mathcal{L}(0,0,0))|\,dr<\infty$ for every $x\in\mathbb{R}$.
\end{enumerate}

Let $X^i$ be a solution to \textnormal{MF-SDE}$^T(\xi^i,B^i,\sigma,Y^i,Z^i)$ such that $X^i\in\mathcal{S}^p_{\mathbb{F}}(0,T)$, and set $\mathcal{L}^i_t:=\mathcal{L}(X^i_t,Y^i_t,Z^i_t)$ for $t\in[0,T]$, $i=1,2$.

\begin{proposition}\label{22lut2}
Assume that $B^1$ satisfies \textnormal{(A1)} and $\int^T_0|B^1-B^2|(r,X^2_r,Y^2_r,Z^2_r,\mathcal{L}^2_r)\,dr<\infty$. Then there exists a constant $c_p$ such that
\begin{equation*}
\begin{split}
\sup_{0\le t\le T}|X^1_t-X^2_t|^p
&\le c_p\Bigg(
|\xi^1-\xi^2|^p
+T^p\sup_{0\le t\le T}|Y^1_t-Y^2_t|^p\\
&\quad
+T^{\frac{1}{2}}\Big(\int_0^T|Z^1_r-Z^2_r|^2\,dr\Big)^{\frac{p}{2}}
+\Big(\int_0^T\mathcal{W}_{p}(\mathcal{L}^1_r,\mathcal{L}^2_r)\,dr\Big)^p\\
&\quad
+\Big(\int_0^T|B^1-B^2|(r,X^2_r,Y^2_r,Z^2_r,\mathcal{L}^2_r)\,dr\Big)^p
\Bigg)
\end{split}
\end{equation*}
\end{proposition}
\begin{proof}
First, $X^1-X^2$ is a finite-variation process. Hence, by the change-of-variables formula, for $t\in[0,T]$,
\begin{equation}\label{22lut1}
\begin{split}
|X^1_t-X^2_t|^2
&=|\xi^1-\xi^2|^2\\
&\quad+2\int_0^t(X^1_r-X^2_r)
\big(B^1(r,X^1_r,Y^1_r,Z^1_r,\mathcal{L}^1_r)
-B^2(r,X^2_r,Y^2_r,Z^2_r,\mathcal{L}^2_r)\big)\,dr.
\end{split}
\end{equation}
By \textnormal{(A1)},
\begin{equation*}
\begin{split}
&(X^1_r-X^2_r)
\big(B^1(r,X^1_r,Y^1_r,Z^1_r,\mathcal{L}^1_r)
-B^2(r,X^2_r,Y^2_r,Z^2_r,\mathcal{L}^2_r)\big)\\
&\quad\le L_x|X^1_r-X^2_r|^2
+L_y|X^1_r-X^2_r||Y^1_r-Y^2_r|
+L_z|X^1_r-X^2_r||Z^1_r-Z^2_r|\\
&\qquad
+\gamma|X^1_r-X^2_r|\mathcal{W}_{p}(\mathcal{L}^1_r,\mathcal{L}^2_r)
+|X^1_r-X^2_r||B^1-B^2|(r,X^2_r,Y^2_r,Z^2_r,\mathcal{L}^2_r)
\end{split}
\end{equation*}
Combining this estimate with \eqref{22lut1} and applying Gronwall's inequality, we obtain, for $t\in[0,T]$,
\begin{equation*}
\begin{split}
|X^1_t-X^2_t|^2
&\le\exp(L_xT)\Bigg(
L_y\int_0^T|X^1_r-X^2_r||Y^1_r-Y^2_r|\,dr\\
&\quad
+L_z\int_0^T|X^1_r-X^2_r||Z^1_r-Z^2_r|\,dr
+\gamma\int_0^T|X^1_r-X^2_r|\mathcal{W}_{p}(\mathcal{L}^1_r,\mathcal{L}^2_r)\,dr\\
&\quad
+\int_0^T|X^1_r-X^2_r||B^1-B^2|(r,X^2_r,Y^2_r,Z^2_r,\mathcal{L}^2_r)\,dr
\Bigg).
\end{split}
\end{equation*}
Hence,
\begin{equation*}
\begin{split}
&\sup_{0\le t\le T}|X^1_t-X^2_t|^p\le c_p\Big(\Big(\int^T_0|X^1_r-X^2_r||Y^1_r-Y^2_r|\,dr\Big)^p+\Big(\int^T_0|X^1_r-X^2_r||Z^1_r-Z^2_r|\,dr\Big)^p\\
&\quad+\Big(\int^T_0|X^1_r-X^2_r|\mathcal{W}_{p}(\mathcal{L}^1_r,\mathcal{L}^2_r)\,dr\Big)^p+\Big(\int^T_0|X^1_r-X^2_r||B^1-B^2|(r,X^2_r,Y^2_r,Z^2_r,\mathcal{L}^2_r)\,dr\Big)^p\Big).
\end{split}
\end{equation*}
The result follows from Young's and H\"older's inequalities.
\end{proof}

\begin{remark}\label{21lut2}
A straightforward calculation shows that, for every $p\in(1,2]$,
\begin{equation}\label{21lut1}
\begin{split}
\Big(\int_0^T\mathcal{W}_{p}(\mathcal{L}^1_r,\mathcal{L}^2_r)\,dr\Big)^p
&\le T_{p}\Bigg(
\mathbb{E}\sup_{0\le t\le T}|X^1_t-X^2_t|^p
+\mathbb{E}\sup_{0\le t\le T}|Y^1_t-Y^2_t|^p\\
&\quad
+\mathbb{E}\Big(\int_0^T|Z^1_r-Z^2_r|^2\,dr\Big)^{\frac{p}{2}}
\Bigg).
\end{split}
\end{equation}
Here $T_p=a_p(T^p+T^{p/2})$ for a constant $a_p$ depending only on $p$. In particular, when $T\le 1$ the factor $T_p$ is bounded by a constant multiple of $T^{1/2}$; this harmless constant is absorbed into the constants appearing below.

Indeed, by coupling the laws with the original random variables and using Minkowski's inequality,
\[
\mathcal{W}_p(\mathcal{L}^1_t,\mathcal{L}^2_t)
\le
\|X^1_t-X^2_t\|_{L^p}
+\|Y^1_t-Y^2_t\|_{L^p}
+\|Z^1_t-Z^2_t\|_{L^p}.
\]
The first two terms are estimated by
\[
\Big(\int_0^T\|X^1_t-X^2_t\|_{L^p}\,dt\Big)^p
\le T^p\,\mathbb{E}\sup_{0\le t\le T}|X^1_t-X^2_t|^p
\]
and similarly for $Y^1-Y^2$. It remains to control the $Z$ term. H\"older's inequality in time and the mixed-norm inequality give
\[
\Big(\int_0^T\|Z^1_t-Z^2_t\|_{L^p}\,dt\Big)^p
\le
T^{p/2}\,
\mathbb{E}\Big(\int_0^T|Z^1_t-Z^2_t|^2\,dt\Big)^{\frac{p}{2}}.
\]
Combining these bounds proves \eqref{21lut1}.
For $p>2$, the last mixed-norm inequality is reversed, and \eqref{21lut1} generally fails for processes in $\mathcal{H}^p_{\mathbb{F}}(0,T)$. Thus the restriction $p\in(1,2]$ is essential under the present assumptions.
\end{remark}

\begin{theorem}\label{22lut4}
Assume that $p\in(1,2]$ and \textnormal{(A1)-(A3)} are satisfied. Then there exists a unique solution $X$ to \textnormal{MF-SDE}$^T(\xi,B,\sigma,Y,Z)$ such that $X\in\mathcal{S}^p_{\mathbb{F}}(0,T)$.
\end{theorem}
\begin{proof}
Define $\Psi:\mathcal{S}^p_{\mathbb{F}}(0,T)\longrightarrow\mathcal{S}^p_{\mathbb{F}}(0,T)$ by $\Psi(\tilde{x})=X$, where $X$ is the unique solution to MF-SDE$^T(\xi,B,\sigma,Y,Z)$ with $B(t,x,y,z)=B(t,x,y,z,\mathcal{L}(\tilde{x},y,z))$, such that $X\in\mathcal{S}^p_{\mathbb{F}}(0,T)$. Existence of this solution follows from \cite[Theorem 3.17]{PR}. Let $X^1,X^2\in\mathcal{S}^p_{\mathbb{F}}(0,T)$ and $\tilde{x}^1,\tilde{x}^2\in\mathcal{S}^p_{\mathbb{F}}(0,T)$ be such that $X^1=\Psi(\tilde{x}^1)$ and $X^2=\Psi(\tilde{x}^2)$. By Proposition \ref{22lut2}, there exists a constant $c_p$ such that
\[
||X^1-X^2||^p_{\mathcal{S}^p}\le c_p\mathbb{E}\Big(\int^T_0\mathcal{W}_{p}(\mathcal{L}(\tilde{x}^1_r,Y_r,Z_r),\mathcal{L}(\tilde{x}^2_r,Y_r,Z_r))\,dr\Big)^p.
\]
Since the two laws differ only in their first component,
\[
\mathcal{W}_{p}(\mathcal{L}(\tilde{x}^1_r,Y_r,Z_r),\mathcal{L}(\tilde{x}^2_r,Y_r,Z_r))
\le \|\tilde{x}^1_r-\tilde{x}^2_r\|_{L^p}.
\]
Consequently, $||X^1-X^2||^p_{\mathcal{S}^p_{\mathbb{F}}(0,T)}\le c||\tilde{x}^1-\tilde{x}^2||^p_{\mathcal{S}^p_{\mathbb{F}}(0,T)}$, where $c=c_pT^p$.

If $c<1$, then $\Psi$ is a contraction. Banach's fixed-point theorem therefore yields $X\in\mathcal{S}^p_{\mathbb{F}}(0,T)$ such that $\Psi(X)=X$, and this process is the unique solution to MF-SDE$^T(\xi,B,\sigma,Y,Z)$. In the general case $c\ge 1$, we divide $[0,T]$ into finitely many sufficiently small intervals. This completes the proof.

\end{proof}

\section{Mean-field RBSDEs with two optional barriers}

Let $X\in\mathcal{S}_{\mathbb{F}}^p(0,T)$. We assume that $L$ and $U$ are $\mathbb{F}$-optional processes of class (D) such that $L_t\le U_t$, $t\in[0,T]$ and $L_T\le h(X_T,\mathcal{L}(X_T))\le U_T$. 

\begin{definition}
We say that a triple $(Y,Z,R)$ of $\mathbb{F}$-adapted processes is a solution to the mean-field reflected backward stochastic differential equation on $[0,T]$ with generator $f$, coefficient $X$, terminal function $h$, lower barrier $L$, and upper barrier $U$ (MF-RBSDE$^T(h,f,X,L,U)$ for short) if
\begin{enumerate}
\item[(a)] $Y$ is a regulated process and $Z\in\mathcal{H}_{\mathbb{F}}(0,T)$,
\item[(b)] $R\in\mathcal{V}_{0,\mathbb{F}}(0,T)$, $L_t\le Y_t\le U_t$, $t\in[0,T]$, and
\begin{equation*}
\begin{split}
&\int^T_0(Y_{r-}-\overrightarrow{L}_r)\,dR^{+,*}_r+\sum_{0\le r<T}(Y_r-L_r)\Delta^+R^+_r\\&
=\int^T_0(\underrightarrow{U}_r-Y_{r-})\,dR^{-,*}_r+\sum_{0\le r<T}(U_r-Y_r)\Delta^+R^-_r=0,
\end{split}
\end{equation*}
where $R=R^+-R^-$ is the Jordan decomposition of $R$,
\item[(c)] $\int^T_0|f(r,X_r,Y_r,Z_r,\mathcal{L}(X_r,Y_r,Z_r))|\,dr<\infty$,
\item[(d)]
\begin{equation*}
\begin{split}
Y_t&=h(X_T,\mathcal{L}(X_T))+\int^{T}_t f(r,X_r,Y_r,Z_r,\mathcal{L}(X_r,Y_r,Z_r))\,dr+R_T-R_t\\
&\quad-\int^{T}_t Z_r\,dW_r,\,t\in[0,T].
\end{split}
\end{equation*}
\end{enumerate}
\end{definition}
We refer to condition (b) as the \textit{minimality condition}.

We use the following hypotheses:
\begin{enumerate}
\item[(B1)] There exist $\lambda_x,\lambda_z,\rho\ge0$ such that
\[
|f(t,x,y,z,\eta)-f(t,x',y,z',\eta')|\le\lambda_x|x-x'|+\lambda_z|z-z'|+\rho\mathcal{W}_{p}(\eta,\eta')
\]
for $t\in[0,T]$, $x,x',y\in\mathbb{R}$, $z,z'\in\mathbb{R}^d$, $\eta,\eta'\in\mathcal{P}_{p}(\mathbb{R}^{2+d})$,
\item[(B2)] There exists $\mu\in\mathbb{R}$ such that
$(y-y')(f(t,x,y,z,\eta)-f(t,x,y',z,\eta))\leq\mu(y-y')^2$ for $t\in[0,T]$, $x,y,y'\in\mathbb{R}$, $z\in\mathbb{R}^d$, and $\eta\in\mathcal{P}_{p}(\mathbb{R}^{d+2})$,
\item[(B3)] For every $(t,x,z,\eta)\in[0,T]\times\mathbb{R}\times\mathbb{R}^d\times\mathcal{P}_{p}(\mathbb{R}^{d+2})$, the mapping $\mathbb{R}\ni y\rightarrow f(t,x,y,z,\eta)$ is continuous,
\item[(B4)] $\int^T_0 |f(r,0,y,0,\mathcal{L}(0,0,0))|\,dr<\infty$ for every $y\in\mathbb{R}$,
\item[(B5)] $h(0,\mathcal{L}(0))\in L^p(\mathcal{F}_T)$, $f(\cdot,0,0,0,\mathcal{L}(0,0,0))\in L^{1,p}_{\mathbb F}(0,T)$,
\item[(B6)] There exists a process $S\in\mathcal{M}_{loc}(0,T)+\mathcal{V}^p_{\mathbb F}(0,T)$ such that
 $L\le S\le U$, $S\in \mathcal S^p_{\mathbb{F}}(0,T)$ and $f(\cdot,0,S,0,\mathcal{L}(0,0,0))\in L^{1,p}_{\mathbb F}(0,T)$.
\end{enumerate}

Let $p\in(1,2]$. Let $(Y^i,Z^i,R^i)$ be a solution to \textnormal{MF-RBSDE}$^T(h^i,f^i,X^i,L,U)$ such that $Y^i\in\mathcal{S}^p_{\mathbb{F}}(0,T)$ and $Z^i\in\mathcal{H}^p_{\mathbb{F}}(0,T)$. Set $\mathcal{L}^i_t:=\mathcal{L}(X^i_t,Y^i_t,Z^i_t)$, $t\in[0,T]$, and $h^i_T:=h^i(X^i_T,\mathcal{L}(X^i_T))$, $i=1,2$. We have the following proposition.

\begin{proposition}\label{20lut19}
Assume that $f^1$ satisfies \textnormal{(B1),(B2)} and that
\[
|f^1-f^2|(\cdot,X^2,Y^2,Z^2,\mathcal{L}^2)\in L^{1,p}_{\mathbb{F}}(0,T).
\]
Then there exists a constant $C_p$ such that
\begin{equation*}
\begin{split}
&\mathbb{E}\sup_{0\le t\le T}|Y^1_t-Y^2_t|^p+\mathbb{E}\Big(\int^T_0|Z^1_r-Z^2_r|^2\,dr\Big)^{\frac{p}{2}}\le C_p\mathbb{E}\Big(|h^1_T-h^2_T|^p+\Big(\int^T_0|X^1_r-X^2_r|\,dr\Big)^p\\
&\quad+\Big(\int^T_0\mathcal{W}_{p}(\mathcal{L}^1_r,\mathcal{L}^2_r)\,dr\Big)^p+\Big(\int^T_0|f^1-f^2|(r,X^2_r,Y^2_r,Z^2_r,\mathcal{L}^2_r)\,dr\Big)^p\Big)
\end{split}
\end{equation*}
\end{proposition}
\begin{proof}
We divide the proof into two steps.\\
\textbf{Step 1.} By \cite[Corollary 5.5]{KRS1}, for $t\in[0,T]$,
\begin{equation}\label{20lut1}
\begin{split}
&|Y^1_t-Y^2_t|^p+\frac{p(p-1)}{2}\int^T_t|Y^1_r-Y^2_r|^{p-2}\mathbf{1}_{\{Y^1_r-Y^2_r\neq 0\}}|Z^1_r-Z^2_r|^2\,dr\le |h^1_T-h^2_T|^p\\
&\quad+p\int^T_t|Y^1_r-Y^2_r|^{p-1}\sgn(Y^1_r-Y^2_r)(f^1(r,X^1_r,Y^1_r,Z^1_r,\mathcal{L}^1_r)-f^2(r,X^2_r,Y^2_r,Z^2_r,\mathcal{L}^2_r))\,dr\\
&\quad+p\int^T_t|Y^1_{r-}-Y^2_{r-}|^{p-1}\sgn(Y^1_{r-}-Y^2_{r-})\,d(R^1_r-R^2_r)^*\\
&\quad+p\sum_{t\le r<T}|Y^1_{r}-Y^2_{r}|^{p-1}\sgn(Y^1_{r}-Y^2_{r})\Delta^+(R^1_r-R^2_r)\\
&\quad-p\int^T_t|Y^1_r-Y^2_r|^{p-1}\sgn(Y^1_r-Y^2_r)(Z^1_r-Z^2_r)\,dW_r.
\end{split}
\end{equation}
By \textnormal{(B1)} and \textnormal{(B2)},
\begin{equation}\label{20lut2}
\begin{split}
&\sgn(Y^1_r-Y^2_r)(f^1(r,X^1_r,Y^1_r,Z^1_r,\mathcal{L}^1_r)-f^2(r,X^2_r,Y^2_r,Z^2_r,\mathcal{L}^2_r))\le\lambda_x|X^1_r-X^2_r|+\mu|Y^1_r-Y^2_r|\\
&\quad+\lambda_z|Z^1_r-Z^2_r|+\rho\mathcal{W}_{p}(\mathcal{L}^1_r,\mathcal{L}^2_r)+|f^1_r-f^2_r|(r,X^2_r,Y^2_r,Z^2_r,\mathcal{L}^2_r).
\end{split}
\end{equation}
Moreover, by the minimality condition,
\begin{equation}\label{20lut3}
\begin{split}
\int^T_0|Y^1_{r-}-Y^2_{r-}|^{p-1}\sgn(Y^1_{r-}-Y^2_{r-})\,d(R^1_r-R^2_r)^*\le 0
\end{split}
\end{equation}
and
\begin{equation}\label{20lut4}
\begin{split}
\sum_{0\le r<T}|Y^1_{r}-Y^2_{r}|^{p-1}\sgn(Y^1_{r}-Y^2_{r})\Delta^+(R^1_r-R^2_r)\le 0.
\end{split}
\end{equation}
Set $f_r:=|f^1_r-f^2_r|(r,X^2_r,Y^2_r,Z^2_r,\mathcal{L}^2_r)$. By \eqref{20lut1}, \eqref{20lut2}, \eqref{20lut3}, and \eqref{20lut4}, for $t\in[0,T]$,
\begin{equation}\label{20lut5}
\begin{split}
&|Y^1_t-Y^2_t|^p+\frac{p(p-1)}{2}\int^T_t|Y^1_r-Y^2_r|^{p-2}\mathbf{1}_{\{Y^1_r-Y^2_r\neq 0\}}|Z^1_r-Z^2_r|^2\,dr\le |h^1_T-h^2_T|^p\\
&\quad+p\int^T_0|Y^1_r-Y^2_r|^{p-1}\big(\lambda_x|X^1_r-X^2_r|+\rho\mathcal{W}_{p}(\mathcal{L}^1_r,\mathcal{L}^2_r)+f_r\big)\,dr\\
&\quad+p\mu\int^T_t|Y^1_r-Y^2_r|^p\,dr+p\lambda_z\int^T_t |Y^1_r-Y^2_r|^{p-1}|Z^1_r-Z^2_r|\,dr\\
&\quad-p\int^T_t|Y^1_r-Y^2_r|^{p-1}\sgn(Y^1_r-Y^2_r)(Z^1_r-Z^2_r)\,dW_r.
\end{split}
\end{equation}
For some $A_p>0$,
\begin{equation*}
\begin{split}
p\lambda_z|Y^1_r-Y^2_r|^{p-1}|Z^1_r-Z^2_r|\,dr\le A_p|Y^1_r-Y^2_r|^p+\frac{p(p-1)}{4}|Y^1_r-Y^2_r|^{p-2}\mathbf{1}_{\{Y^1_r-Y^2_r\neq 0\}}|Z^1_r-Z^2_r|^2
\end{split}
\end{equation*}
Combining this with \eqref{20lut5} gives
\begin{equation}\label{20lut6}
\begin{split}
&|Y^1_t-Y^2_t|^p+\frac{p(p-1)}{4}\int^T_t|Y^1_r-Y^2_r|^{p-2}\mathbf{1}_{\{Y^1_r-Y^2_r\neq 0\}}|Z^1_r-Z^2_r|^2\,dr\le |h^1_T-h^2_T|^p\\
&\quad+p\int^T_0|Y^1_r-Y^2_r|^{p-1}\big(\lambda_x|X^1_r-X^2_r|+\rho\mathcal{W}_{p}(\mathcal{L}^1_r,\mathcal{L}^2_r)+f_r\big)\,dr\\
&\quad+(p\mu+A_p)\int^T_t|Y^1_r-Y^2_r|^p\,dr-p\int^T_t|Y^1_r-Y^2_r|^{p-1}\sgn(Y^1_r-Y^2_r)(Z^1_r-Z^2_r)\,dW_r,
\end{split}
\end{equation}
for $t\in[0,T]$. By the Burkholder-Davis-Gundy inequality, $\Big\{\int^t_0|Y^1_r-Y^2_r|^{p-1}\sgn(Y^1_r-Y^2_r)(Z^1_r-Z^2_r)\,dW_r\Big\}_{t\in[0,T]}$ is a uniformly integrable martingale. Indeed, Young's inequality gives
\begin{equation*}
\begin{split}
\mathbb{E}\Big(\int^T_0|Y^1_r-Y^2_r|^{2p-2}|Z^1_r-Z^2_r|^2\,dr\Big)^{\frac{1}{2}}\le\frac{p-1}{p}\mathbb{E}\sup_{0\le t\le T}|Y^1_t-Y^2_t|^p+\frac{1}{p}\mathbb{E}\Big(\int^T_0|Z^1_r-Z^2_r|^2\,dr\Big)^{\frac{p}{2}}.
\end{split}
\end{equation*} 
Taking expectation in \eqref{20lut6} with $t=0$, we obtain
\begin{equation}\label{20lut7}
\begin{split}
&\frac{p(p-1)}{4}\mathbb{E}\int^T_0|Y^1_r-Y^2_r|^{p-2}\mathbf{1}_{\{Y^1_r-Y^2_r\neq 0\}}|Z^1_r-Z^2_r|^2\,dr\le \mathbb{E}\Big(|h^1_T-h^2_T|^p\\
&\quad+p\int^T_0|Y^1_r-Y^2_r|^{p-1}(\lambda_x|X^1_r-X^2_r|+\rho\mathcal{W}_{p}(\mathcal{L}^1_r,\mathcal{L}^2_r))+f_r)\,dr\\
&\quad+(p\mu+A_p)\int^T_t|Y^1_r-Y^2_r|^p\,dr\Big),\quad t\in[0,T].
\end{split}
\end{equation}
Moreover, Gronwall's lemma yields
\begin{equation*}
\begin{split}
&\mathbb{E}|Y^1_t-Y^2_t|^p\le\exp((p\mu+A_p)T)p\mathbb{E}\Big(|h^1_T-h^2_T|^p+\int^T_0|Y^1_r-Y^2_r|^{p-1}\big(\lambda_x|X^1_r-X^2_r|\\
&\quad+\rho\mathcal{W}_{p}(\mathcal{L}^1_r,\mathcal{L}^2_r)+f_r\big)\,dr\Big).
\end{split}
\end{equation*}
Thus, using the preceding estimate and \eqref{20lut7}, for $t\in[0,T]$,
\begin{equation}\label{20lut9}
\begin{split}
&\mathbb{E}|Y^1_t-Y^2_t|^p+\mathbb{E}\int^T_0|Y^1_r-Y^2_r|^{p-2}\mathbf{1}_{\{Y^1_r-Y^2_r\neq 0\}}|Z^1_r-Z^2_r|^2\le C_p\mathbb{E}\Big(|h^1_T-h^2_T|^p\\
&\quad+\int^T_0|Y^1_r-Y^2_r|^{p-1}\big(\lambda_x|X^1_r-X^2_r|+\rho\mathcal{W}_{p}(\mathcal{L}^1_r,\mathcal{L}^2_r)+f_r\big)\,dr\Big).
\end{split}
\end{equation}
for some $C_p>0$. Let $X:=|h^1_T-h^2_T|^p+\int^T_0|Y^1_r-Y^2_r|^{p-1}\big(\lambda_x|X^1_r-X^2_r|+\rho\mathcal{W}_{p}(\mathcal{L}^1_r,\mathcal{L}^2_r)+f_r\big)\,dr$. Applying \eqref{20lut6}, \eqref{20lut9}, and the Burkholder-Davis-Gundy inequality again, we obtain
\begin{equation}\label{20lut10}
\begin{split}
\mathbb{E}\sup_{0\le t\le T}|Y^1_t-Y^2_t|^p\le \mathbb{E}X+\kappa_p\mathbb{E}\Big(\int^T_0|Y^1_r-Y^2_r|^{2p-2}|Z^1_r-Z^2_r|^2\,dr\Big)^{\frac{1}{2}}
\end{split}
\end{equation}
for some $\kappa_p>0$. Note that
\begin{equation*}
\begin{split}
&\kappa_p\mathbb{E}\Big(\int^T_0|Y^1_r-Y^2_r|^{2p-2}|Z^1_r-Z^2_r|^2\,dr\Big)^{\frac{1}{2}}\le \frac{1}{2}\mathbb{E}\sup_{0\le t\le T}|Y^1_t-Y^2_t|^p\\
&\quad+\frac{\kappa^2_p}{2}\mathbb{E}\int^T_0|Y^1_r-Y^2_r|^{p-2}\mathbf{1}_{\{Y^1_r-Y^2_r\neq 0\}}|Z^1_r-Z^2_r|^2\,dr
\end{split}
\end{equation*}
Therefore, by \eqref{20lut9} and \eqref{20lut10},
\begin{equation*}
\mathbb{E}\sup_{0\le t\le T}|Y^1_t-Y^2_t|^p\le C_p\mathbb{E}X.
\end{equation*}
Next, Young's inequality gives
\begin{equation*}
\begin{split}
&p\rho\,\mathbb{E}\int^T_0|Y^1_r-Y^2_r|^{p-1}\mathcal{W}_{p}(\mathcal{L}^1_r,\mathcal{L}^2_r)\,dr\le (p\rho)^{\frac{p-1}{p}}\big(\frac{1}{\alpha}\big)^{\frac{p-1}{p}}\mathbb{E}\sup_{0\le t\le T}|Y^1_t-Y^2_t|^p\\
&\quad+\alpha^p\Big(\mathbb{E}\int^T_0\mathcal{W}_{p}(\mathcal{L}^1_r,\mathcal{L}^2_r)\,dr\Big)^p.
\end{split}
\end{equation*}
Treating the other terms in $X$ analogously and choosing the constants appropriately (for instance, $\alpha=\lambda_xp 2^{\frac{p}{p-1}}C^{\frac{p}{p-1}}_p$ above), we obtain
\begin{equation}\label{20lut18}
\begin{split}
\mathbb{E}\sup_{0\le t\le T}|Y^1_t-Y^2_t|^p
&\le C_p\mathbb{E}\Bigg(
|h^1_T-h^2_T|^p
+\Big(\int_0^T|X^1_r-X^2_r|\,dr\Big)^p\\
&\quad
+\Big(\int_0^T\mathcal{W}_{p}(\mathcal{L}^1_r,\mathcal{L}^2_r)\,dr\Big)^p\\
&\quad
+\Big(\int_0^T|f^1_r-f^2_r|(r,X^2_r,Y^2_r,Z^2_r,\mathcal{L}^2_r)\,dr\Big)^p
\Bigg)
\end{split}
\end{equation}

\textbf{Step 2.} By \cite[Corollary 5.5]{KRS1}, for $t\in[0,T]$,
\begin{equation}\label{20lut12}
\begin{split}
\int^T_t|Z^1_r-Z^2_r|^2\,dr
&\le |h^1_T-h^2_T|^2\\
&\quad+2\int^T_t(Y^1_r-Y^2_r)
\big(f^1(r,X^1_r,Y^1_r,Z^1_r,\mathcal{L}^1_r)\\
&\qquad\qquad\qquad
-f^2(r,X^2_r,Y^2_r,Z^2_r,\mathcal{L}^2_r)\big)\,dr\\
&\quad+2\int^T_t(Y^1_{r-}-Y^2_{r-})\,d(R^1_r-R^2_r)^*\\
&\quad+2\sum_{t\le r<T}(Y^1_r-Y^2_r)\Delta^+(R^1_r-R^2_r)\\
&\quad-2\int^T_t(Y^1_r-Y^2_r)(Z^1_r-Z^2_r)\,dW_r.
\end{split}
\end{equation}
By \textnormal{(B1)} and \textnormal{(B2)},
\begin{equation}\label{20lut13}
\begin{split}
&(Y^1_r-Y^2_r)
\big(f^1(r,X^1_r,Y^1_r,Z^1_r,\mathcal{L}^1_r)
-f^2(r,X^2_r,Y^2_r,Z^2_r,\mathcal{L}^2_r)\big)\\
&\quad\le
\lambda_x|Y^1_r-Y^2_r||X^1_r-X^2_r|
+\mu|Y^1_r-Y^2_r|^2
+\frac{1}{8}|Z^1_r-Z^2_r|^2\\
&\qquad
+8\lambda^2_z|Y^1_r-Y^2_r|^2
+\rho|Y^1_r-Y^2_r|\mathcal{W}_{p}(\mathcal{L}^1_r,\mathcal{L}^2_r)\\
&\qquad
+|Y^1_r-Y^2_r||f^1_r-f^2_r|(r,X^2_r,Y^2_r,Z^2_r,\mathcal{L}^2_r).
\end{split}
\end{equation}
By the minimality condition,
\begin{equation}\label{20lut14}
2\int^T_t(Y^1_{r-}-Y^2_{r-})\,d(R^1_r-R^2_r)^*+2\sum_{t\le r<T}(Y^1_r-Y^2_r)\Delta^+(R^1_r-R^2_r)\le 0.
\end{equation}
Thus, by \eqref{20lut12}, \eqref{20lut13}, \eqref{20lut14}, and the Burkholder-Davis-Gundy inequality,
\begin{equation}\label{20lut15}
\begin{split}
&\mathbb{E}\Big(\int^T_0|Z^1_r-Z^2_r|^2\,dr\Big)^{\frac{p}{2}}\le D_p\Big(\mathbb{E}|h^1_T-h^2_T|^p+\mathbb{E}\Big(\int^T_0|Y^1_r-Y^2_r|(\lambda_x|X^1_r-X^2_r|\\
&\quad+\rho\mathcal{W}_{p}(\mathcal{L}^1_r,\mathcal{L}^2_r)+f_r)\,dr\Big)^{\frac{p}{2}}+\mathbb{E}\Big(\int^T_0|Y^1_r-Y^2_r|^2\,dr\Big)^{\frac{p}{2}}+\mathbb{E}\Big(\int^T_0|Y^1_r-Y^2_r|^2|Z^1_r-Z^2_r|^2\,dr\Big)^{\frac{p}{4}}\Big).
\end{split}
\end{equation}
By Young's inequality,
\begin{equation}\label{20lut16}
\begin{split}
&\mathbb{E}\Big(\int^T_0|Y^1_r-Y^2_r|(\lambda_x|X^1_r-X^2_r|+\rho\mathcal{W}_{p}(\mathcal{L}^1_r,\mathcal{L}^2_r)+f_r)\,dr\Big)^{\frac{p}{2}}\le \mathbb{E}\sup_{0\le t\le T}|Y^1_t-Y^2_t|^p\\
&\quad+D_p\mathbb{E}\Big(\Big(\int^T_0|X^1_r-X^2_r|\,dr\Big)^p+\Big(\int^T_0\mathcal{W}_{p}(\mathcal{L}^1_r,\mathcal{L}^2_r)\,dr\Big)^p\\
&\quad+\Big(\int^T_0|f^1_r-f^2_r|(r,X^2_r,Y^2_r,Z^2_r,\mathcal{L}^2_r)\,dr\Big)^p\Big)
\end{split}
\end{equation}
for some $D_p>0$ and
\begin{equation}\label{20lut17}
\begin{split}
\mathbb{E}\Big(\int^T_0|Y^1_r-Y^2_r|^2|Z^1_r-Z^2_r|^2\,dr\Big)^{\frac{p}{4}}\le\frac{1}{2}\mathbb{E}\Big(\int^T_0|Z^1_r-Z^2_r|^2\,dr\Big)^{\frac{p}{2}}+\gamma_p\mathbb{E}\sup_{0\le t\le T}|Y^1_t-Y^2_t|^p.
\end{split}
\end{equation}
Finally, \eqref{20lut15}, \eqref{20lut16}, \eqref{20lut17}, and \eqref{20lut18} imply the desired estimate.
\end{proof}

\begin{remark}
A similar result for classical RBSDEs with optional barriers can be found, for instance, in \cite[Proposition 3.7]{KRS} when $f$ is independent of $z$. In the present mean-field setting, however, the change-of-variables property used in \cite[Remark 3.2]{KRS1} does not apply. We therefore use a slightly different argument based on Gronwall's lemma.
\end{remark}

\begin{theorem}\label{22lut3}
Assume that $p\in(1,2]$ and \textnormal{(B1)-(B6)} are satisfied. Then there exists a unique solution $(Y,Z,R)$ to \textnormal{MF-RBSDE}$^T(h,f,X,L,U)$ such that $Y\in\mathcal{S}^p_{\mathbb{F}}(0,T)$, $Z\in\mathcal{H}^p_{\mathbb{F}}(0,T)$, $R\in\mathcal{V}^p_{0,\mathbb{F}}(0,T)$.
\end{theorem}
\begin{proof}
Consider the space $\mathcal{S}^p_{\mathbb{F}}(0,T)\oplus\mathcal{H}^p_{\mathbb{F}}(0,T)$ equipped with the norm
\[
||(Y,Z)||_{\mathcal{S}^p\oplus\mathcal{H}^p}:=\Big(\mathbb{E}\sup_{0\le t\le T}|Y_t|^p+\mathbb{E}\Big(\int^T_0|Z_r|^2\,dr\Big)^{\frac{p}{2}}\Big)^{\frac{1}{p}}.
\]
Define $\Phi:\mathcal{S}^p_{\mathbb{F}}(0,T)\oplus\mathcal{H}^p_{\mathbb{F}}(0,T)\longrightarrow\mathcal{S}^p_{\mathbb{F}}(0,T)\oplus\mathcal{H}^p_{\mathbb{F}}(0,T)$ by $\Phi(G,H)=(Y,Z)$, where $(Y,Z,R)$ is the unique solution to MF-RBSDE$^T(h,f,X,L,U)$ with $f(t,x,y,z)=f(t,x,y,z,\mathcal{L}(x,G_t,H_t))$, such that $Y\in\mathcal{S}^p_{\mathbb{F}}(0,T)$, $Z\in\mathcal{H}^p_{\mathbb{F}}(0,T)$, and $R\in\mathcal{V}^p_{0,\mathbb{F}}(0,T)$. Existence of this solution follows from \cite[Theorem 3.9]{KRS}. Let $(Y^1,Z^1),(Y^2,Z^2)\in\mathcal{S}^p_{\mathbb{F}}(0,T)\oplus\mathcal{H}^p_{\mathbb{F}}(0,T)$ and $(G^1,H^1),(G^2,H^2)\in\mathcal{S}^p_{\mathbb{F}}(0,T)\oplus\mathcal{H}^p_{\mathbb{F}}(0,T)$ be such that $(Y^1,Z^1)=\Phi(G^1,H^1)$ and $(Y^2,Z^2)=\Phi(G^2,H^2)$. By Proposition \ref{20lut19}, there exists a constant $C_p$ such that
\[
||(Y^1-Y^2,Z^1-Z^2)||^p_{\mathcal{S}^p\oplus\mathcal{H}^p}\le C_p\mathbb{E}\Big(\int^T_0\mathcal{W}_{p}(\mathcal{L}(X_r,G^1_r,H^1_r),\mathcal{L}(X_r,G^2_r,H^2_r))\,dr\Big)^p.
\]
By Remark \ref{21lut2},
\[
||(Y^1-Y^2,Z^1-Z^2)||^p_{\mathcal{S}^p\oplus\mathcal{H}^p}\le C||(G^1-G^2,H^1-H^2)||^p_{\mathcal{S}^p\oplus\mathcal{H}^p},
\]
where $C=C_p\cdot T_{p}$.

If $C<1$, then $\Phi$ is a contraction. Banach's fixed-point theorem therefore yields $(\tilde{Y},\tilde{Z})$ such that $\Phi(\tilde{Y},\tilde{Z})=(\tilde{Y},\tilde{Z})$. Let $\tilde{R}$ be the reflecting process associated with this fixed point in the definition of $\Phi$. Then $(\tilde{Y},\tilde{Z},\tilde{R})$ is the unique solution to MF-RBSDE$^T(h,f,X,L,U)$. In the general case $C\ge 1$, we divide $[0,T]$ into finitely many sufficiently small intervals. This completes the proof.

\end{proof}

\section{Mean-field FBSDEs with two optional barriers}

\begin{definition}
We say that a quadruple $(X,Y,Z,R)$ of $\mathbb{F}$-adapted processes is a solution to the mean-field forward-backward stochastic differential equation with reflection on $[0,T]$, with generators $B$ and $f$, diffusion coefficient $\sigma$, initial condition $\xi$, terminal function $h$, lower barrier $L$, and upper barrier $U$ (MF-FBSDER$^T(\xi,h,B,\sigma,f,L,U)$ for short), if
\begin{enumerate}[(a)]
\item $X,Y$ are regulated processes and $Z\in\mathcal{H}_{\mathbb{F}}(0,T)$,
\item $R\in\mathcal{V}_{0,\mathbb{F}}(0,T)$, $L_t\le Y_t\le U_t$, $t\in[0,T]$, and
\begin{equation*}
\begin{split}
&\int^T_0(Y_{r-}-\overrightarrow{L}_r)\,dR^{+,*}_r+\sum_{0\le r<T}(Y_r-L_r)\Delta^+R^+_r\\&
=\int^T_0(\underrightarrow{U}_r-Y_{r-})\,dR^{-,*}_r+\sum_{0\le r<T}(U_r-Y_r)\Delta^+R^-_r=0,
\end{split}
\end{equation*}
where $R=R^+-R^-$ is the Jordan decomposition of $R$,
\item $\int^T_0|B(r,X_r,Y_r,Z_r,\mathcal{L}(X_r,Y_r,Z_r))|+|f(r,X_r,Y_r,Z_r,\mathcal{L}(X_r,Y_r,Z_r))|\,dr<\infty$,
\item $X_t=\xi+\int^t_0 B(r,X_r,Y_r,Z_r,\mathcal{L}(X_r,Y_r,Z_r))\,dr+\sigma W_t$, $t\in[0,T]$,
\item $Y_t=h(X_T,\mathcal{L}(X_T))+\int^T_t f(r,X_r,Y_r,Z_r,\mathcal{L}(X_r,Y_r,Z_r))\,dr+R_T-R_t-\int^T_t Z_r \,dW_r$, $t\in[0,T]$.
\end{enumerate}
\end{definition}

We impose the following additional assumption.
\begin{enumerate}
\item[(B1')] There exists $l>0$ such that $|h(x,\eta)-h(x',\eta')|\le l(|x-x'|+\mathcal{W}_p(\eta,\eta'))$, for $x,x'\in\mathbb{R}$ and $\eta,\eta'\in\mathcal{P}_p(\mathbb{R})$.
\end{enumerate}

For the remainder of the paper, we assume that \textnormal{(A1)}, \textnormal{(A2)}, \textnormal{(B1')}, and \textnormal{(B1)--(B6)} are in force.

Let $p\in(1,2]$. For $x\in\mathcal{S}^p_{\mathbb{F}}(0,T)$, define $\psi(x)=(Y,Z)$, where $(Y,Z,R)$ is the solution to MF-RBSDE$^T(h,f,x,L,U)$ satisfying $Y\in\mathcal{S}^p_{\mathbb{F}}(0,T)$, $Z\in\mathcal{H}^p_{\mathbb{F}}(0,T)$, and $R\in\mathcal{V}^p_{0,\mathbb{F}}(0,T)$. Next, for $(Y,Z)$, define $\phi(Y,Z)=X$, where $X$ is the solution to MF-SDE$^T(\xi,B,\sigma,Y,Z)$ satisfying $X\in\mathcal{S}^p_{\mathbb{F}}(0,T)$. By Theorems \ref{22lut3} and \ref{22lut4}, respectively, $\psi$ and $\phi$ are well-defined.

\begin{theorem}[Short-time existence]
Let $p\in(1,2]$, and let $c_p$ and $C_p$ be the constants appearing in \textnormal{Propositions \ref{22lut2} and \ref{20lut19}}, respectively. Assume that $T\le 1$, $c_pT^{\frac{1}{2}}<1$, $C_pT^{\frac{1}{2}}<1$, and
\begin{equation*}
\Big(\frac{2c_p}{1-c_pT^{\frac{1}{2}}}\Big)\Big(\frac{C_p}{1-C_pT^{\frac{1}{2}}}\Big)(2l+2T^{\frac{1}{2}})\cdot T^{\frac{1}{2}}=:M\cdot T^{\frac{1}{2}}<1.
\end{equation*}
Then there exists a unique solution $(X,Y,Z,R)$ of \textnormal{MF-FBSDER} such that $X,Y\in\mathcal{S}^p_{\mathbb{F}}(0,T)$, $Z\in\mathcal{H}^p_{\mathbb{F}}(0,T)$, and $R\in\mathcal{V}^p_{0,\mathbb{F}}(0,T)$.
\end{theorem}
\begin{proof}
We show that $\phi\circ\psi$ is a contraction. By Remark \ref{21lut2}, after absorbing the harmless constant in $T_p\lesssim T^{\frac{1}{2}}$ into $c_p$ and $C_p$, we may use $T_{p}\le T^{\frac{1}{2}}$; also $T\le T^{\frac{1}{2}}$. By \textnormal{(B1')} and Proposition \ref{20lut19},
\begin{equation}\label{22lut5}
\begin{split}
&\mathbb{E}\sup_{0\le t\le T}|Y^1_t-Y^2_t|^p+\mathbb{E}\Big(\int^T_0|Z^1_r-Z^2_r|^2\,dr\Big)^{\frac{p}{2}}\le \Big(\frac{C_p}{1-C_pT^{\frac{1}{2}}}\Big)(2l+2T^{\frac{1}{2}})\mathbb{E}\sup_{0\le t\le T}|x^1_t-x^2_t|^p
\end{split}
\end{equation}
Furthermore, Proposition \ref{22lut2} gives
\begin{equation*}
\begin{split}
\mathbb{E}\sup_{0\le t\le T}|X^1_t-X^2_t|^p\le \Big(\frac{2c_p}{1-c_pT^{\frac{1}{2}}}\Big)T^{\frac{1}{2}}&\mathbb{E}\sup_{0\le t\le T}|Y^1_t-Y^2_t|^p+\mathbb{E}\Big(\int^T_0|Z^1_r-Z^2_r|^2\,dr\Big)^{\frac{p}{2}}
\end{split}
\end{equation*}
Therefore, by the preceding estimate and \eqref{22lut5},
\[
\mathbb{E}\sup_{0\le t\le T}|X^1_t-X^2_t|^p\le M\cdot T^{\frac{1}{2}}\mathbb{E}\sup_{0\le t\le T}|x^1_t-x^2_t|^p.
\]
Hence, $\phi\circ\psi$ is a contraction on $(S^p_{\mathbb{F}}(0,T),||\cdot||_{\mathcal{S}^p_{\mathbb{F}}(0,T)})$. Banach's fixed-point theorem yields a unique process $X\in\mathcal{S}^p_{\mathbb{F}}(0,T)$ such that $\psi(X)=(Y,Z)\in\mathcal{S}^p_{\mathbb{F}}(0,T)\otimes\mathcal{H}^p_{\mathbb{F}}(0,T)$ and $\phi(Y,Z)=X$. By the definition of $\psi$, there exists $R\in\mathcal{V}^p_{0,\mathbb{F}}(0,T)$ such that $(X,Y,Z,R)$ is a solution to MF-FBSDER$^T(\xi,h,B,\sigma,f,L,U)$.
\end{proof}

We now impose a second additional assumption.
\begin{enumerate}
\item[(C)] There exist constants $K,\alpha,\beta>0$ such that
\begin{enumerate}[(i)]
\item
\[
-2\mu-\lambda^2_x-2\lambda^2_z-\alpha\rho^2>K>2L_x+\beta(L^2_y+L^2_z+\gamma^2)
\]
\item
\[
\frac{T\vee 1}{\alpha}<\frac{1}{20},\quad\frac{T\vee 1}{\beta}<1,\quad\frac{2(T\vee 1)\cdot 20(2l+(T\vee 1)+\frac{1}{20})}{\beta-(T\vee 1)}<1. 
\]
\end{enumerate}
\end{enumerate}

\begin{theorem}
Let $p=2$ and assume that \textnormal{(C)} is satisfied. Then there exists a unique solution $(X,Y,Z,R)$ of \textnormal{MF-FBSDER} such that $X,Y\in\mathcal{S}^2_{\mathbb{F}}(0,T)$, $Z\in\mathcal{H}^2_{\mathbb{F}}(0,T)$, and $R\in\mathcal{V}^2_{0,\mathbb{F}}(0,T)$.
\end{theorem}
\begin{proof}
Under these additional assumptions, we again show that $\phi\circ\psi$ is a contraction.
By It\^o's formula, for $K>0$,
\begin{equation}\label{24czerwca1}
\begin{split}
&e^{-Kt}|Y^1_t-Y^2_t|^2+\int^T_t e^{-Kr}|Z^1_r-Z^2_r|^2\,dr+(-K)\int^T_t e^{-Kr}|Y^1_r-Y^2_r|^2\,dr\\
&\quad\le|h(x^1_T,\mathcal{L}(x^1_T))-h(x^2_T,\mathcal{L}(x^2_T))|^2+2\int^T_t e^{-Kr}(Y^1_r-Y^2_r)(f(r,x^1_r,Y^1_r,Z^1_r,\mathcal{L}(x^1_r,Y^1_r,Z^1_r))\\
&\quad\quad-f(r,x^2_r,Y^2_r,Z^2_r,\mathcal{L}(x^2_r,Y^2_r,Z^2_r)))\,dr+2\int^T_t e^{-Kr}(Y^1_{r-}-Y^2_{r-})\,d(R^1_r-R^2_r)^*\\
&\quad\quad+\sum_{t\le r<T}e^{-Kr}(Y^1_r-Y^2_r)\,\Delta^+(R^1_r-R^2_r)-2\int^T_t e^{-Kr}(Y^1_r-Y^2_r)(Z^1_r-Z^2_r)\,dW_r 
\end{split}
\end{equation}
By \textnormal{(B1)} and \textnormal{(B2)},
\begin{equation}\label{24czerwca2}
\begin{split}
&2(Y^1_r-Y^2_r)(f(r,x^1_r,Y^1_r,Z^1_r,\mathcal{L}(x^1_r,Y^1_r,Z^1_r))-f(r,x^2_r,Y^2_r,Z^2_r,\mathcal{L}(x^2_r,Y^2_r,Z^2_r)))\\
&\quad\le 2\mu|Y^1_r-Y^2_r|^2+\lambda^2_x|Y^1_r-Y^2_r|^2+|x^1_r-x^2_r|^2+2\lambda^2_z|Y^1_r-Y^2_r|^2+\frac{1}{2}|Z^1_r-Z^2_r|^2\\
&\quad+\alpha\rho^2|Y^1_r-Y^2_r|^2+\frac{1}{\alpha}\mathcal{W}^2_2(\mathcal{L}(x^1_r,Y^1_r,Z^1_r),\mathcal{L}(x^2_r,Y^2_r,Z^2_r))
\end{split}
\end{equation}
Moreover, by the minimality condition,
\begin{equation}\label{24czerwca3}
\int^T_t e^{-Kr}(Y^1_{r-}-Y^2_{r-})\,d(R^1_r-R^2_r)^*+\sum_{t\le r<T}e^{-Kr}(Y^1_r-Y^2_r)\,\Delta^+(R^1_r-R^2_r)\le 0.
\end{equation}
Using \eqref{24czerwca1}, \eqref{24czerwca2}, the assumption $-2\mu-\lambda^2_x-2\lambda^2_z-\alpha\rho^2>K$, and the fact that $\Big\{\int^{\cdot}_0 e^{-Kr}(Y^1_r-Y^2_r)(Z^1_r-Z^2_r)\,dW_r\Big\}$ is a martingale, we obtain
\begin{equation}\label{24czerwca4}
\begin{split}
&E\int^T_0 e^{-Kr}|Z^1_r-Z^2_r|^2\,dr\le2\Bigg(|h(x^1_T,\mathcal{L}(x^1_T))-h(x^2_T,\mathcal{L}(x^2_T))|^2+E\int^T_0 e^{-Kr}|x^1_r-x^2_r|^2\,dr\\
&\quad+\frac{1}{\alpha}E\int^T_0 e^{-Kr}\mathcal{W}^2_2(\mathcal{L}(x^1_r,Y^1_r,Z^1_r),\mathcal{L}(x^2_r,Y^2_r,Z^2_r))\,dr\Bigg).
\end{split}
\end{equation}
Similarly, \eqref{24czerwca1}, \eqref{24czerwca2}, \eqref{24czerwca3}, and the Burkholder-Davis-Gundy inequality give
\begin{equation}\label{24czerwca5}
\begin{split}
&E\sup_{0\le t\le T}e^{-Kt}|Y^1_t-Y^2_t|^2\le|h(x^1_T,\mathcal{L}(x^1_T))-h(x^2_T,\mathcal{L}(x^2_T))|^2+E\int^T_0 e^{-Kr}|x^1_r-x^2_r|^2\,dr\\
&\quad+\frac{1}{\alpha}E\int^T_0 e^{-Kr}\mathcal{W}^2_2(\mathcal{L}(x^1_r,Y^1_r,Z^1_r),\mathcal{L}(x^2_r,Y^2_r,Z^2_r))\,dr\\
&\quad+E\sup_{0\le t\le T}\Big|\int^T_t e^{-Kr}(Y^1_r-Y^2_r)(Z^1_r-Z^2_r)\,dW_r\Big|\le|h(x^1_T,\mathcal{L}(x^1_T))-h(x^2_T,\mathcal{L}(x^2_T))|^2\\
&\quad+E\int^T_0 e^{-Kr}|x^1_r-x^2_r|^2\,dr+\frac{1}{\alpha}E\int^T_0 e^{-Kr}\mathcal{W}^2_2(\mathcal{L}(x^1_r,Y^1_r,Z^1_r),\mathcal{L}(x^2_r,Y^2_r,Z^2_r))\,dr\\
&\quad+\frac{1}{2}E\sup_{0\le t\le T}e^{-Kt}|Y^1_t-Y^2_t|^2+2E\int^T_0 e^{-Kr}|Z^1_r-Z^2_r|^2\,dr.
\end{split}
\end{equation}
Combining \eqref{24czerwca4} and \eqref{24czerwca5} gives
\begin{equation}\label{24czerwca6}
\begin{split}
&E\sup_{0\le t\le T}e^{-Kt}|Y^1_t-Y^2_t|^2+E\int^T_0 e^{-Kr}|Z^1_r-Z^2_r|^2\,dr\le10\Bigg(|h(x^1_T,\mathcal{L}(x^1_T))-h(x^2_T,\mathcal{L}(x^2_T))|^2\\
&\quad+E\int^T_0 e^{-Kr}|x^1_r-x^2_r|^2\,dr+\frac{1}{\alpha}E\int^T_0 e^{-Kr}\mathcal{W}^2_2(\mathcal{L}(x^1_r,Y^1_r,Z^1_r),\mathcal{L}(x^2_r,Y^2_r,Z^2_r))\,dr\Bigg).
\end{split}
\end{equation}
Since
\begin{equation}\label{24czerwca7}
\begin{split}
&E\int^T_0 e^{-Kr}\mathcal{W}^2_2(\mathcal{L}(x^1_r,Y^1_r,Z^1_r),\mathcal{L}(x^2_r,Y^2_r,Z^2_r))\,dr\le(T\vee 1)\Big(E\sup_{0\le t\le T}e^{-Kt}|Y^1_t-Y^2_t|^2\\
&\quad+E\int^T_0 e^{-Kr}|Z^1_r-Z^2_r|^2\,dr+E\sup_{0\le t\le T}e^{-Kt}|x^1_t-x^2_t|^2\Big)
\end{split}
\end{equation}
and by \textnormal{(B1')},
\begin{equation}\label{24czerwca8}
\begin{split}
\mathbb{E}e^{-KT}|h(x^1_T,\mathcal{L}(x^1_T))-h(x^2_T,\mathcal{L}(x^2_T))|^2\le 2l\mathbb{E}\sup_{0\le t\le T}e^{-Kt}|x^1_t-x^2_t|^2,
\end{split}
\end{equation}
Using \eqref{24czerwca6}, \eqref{24czerwca7}, and \eqref{24czerwca8}, and since $\frac{T\vee 1}{\alpha}\le\frac{1}{20}$, we have
\begin{equation}\label{24czerwca9}
\begin{split}
&E\sup_{0\le t\le T}e^{-Kt}|Y^1_t-Y^2_t|^2+E\int^T_0 e^{-Kr}|Z^1_r-Z^2_r|^2\,dr\le 20\big(2l+(T\vee 1)+\frac{1}{20}\big)E\sup_{0\le t\le T}e^{-Kt}|x^1_t-x^2_t|^2.
\end{split}
\end{equation}
Next, another application of It\^o's formula gives
\begin{equation}\label{27czerwca1}
\begin{split}
&e^{-Kt}|X^1_t-X^2_t|^2+K\int^t_0 e^{-Kr}|X^1_r-X^2_r|^2\,dr\\
&\quad=2\int^t_0e^{-Kr}(X^1_r-X^2_r)
\big(B(r,X^1_r,Y^1_r,Z^1_r,\mathcal{L}(X^1_r,Y^1_r,Z^1_r))\\
&\quad\quad-B(r,X^2_r,Y^2_r,Z^2_r,\mathcal{L}(X^2_r,Y^2_r,Z^2_r))\big)\,dr.
\end{split}
\end{equation}
By \textnormal{(A1)},
\begin{equation}\label{27czerwca2}
\begin{split}
&2(X^1_r-X^2_r)(B(r,X^1_r,Y^1_r,Z^1_r,\mathcal{L}(X^1_r,Y^1_r,Z^1_r))-B(r,X^2_r,Y^2_r,Z^2_r,\mathcal{L}(X^2_r,Y^2_r,Z^2_r))\\
&\quad\le 2L_x|X^1_r-X^2_r|^2+\beta L^2_y|X^1_r-X^2_r|^2+\frac{1}{\beta}|Y^1_r-Y^2_r|^2+\beta L^2_z|X^1_r-X^2_r|^2+\frac{1}{\beta}|Z^1_r-Z^2_r|^2\\
&\quad+\beta \gamma^2|X^1_r-X^2_r|^2+\frac{1}{\beta}\mathcal{W}^2_2(\mathcal{L}(X^1_r,Y^1_r,Z^1_r),\mathcal{L}(X^2_r,Y^2_r,Z^2_r))
\end{split}
\end{equation}
Combining \eqref{27czerwca1} and \eqref{27czerwca2}, and using $K>2L_x+\beta(L^2_y+L^2_z+\gamma^2)$, we obtain
\begin{equation}\label{27czerwca3}
\begin{split}
&E\sup_{0\le t\le T}e^{-Kt}|X^1_t-X^2_t|^2\\
&\quad\le \frac{T\vee 1}{\beta}
\Big(E\sup_{0\le t\le T}e^{-Kt}|Y^1_t-Y^2_t|^2
+E\int_0^T e^{-Kr}|Z^1_r-Z^2_r|^2\,dr\Big)\\
&\qquad
+\frac{1}{\beta}E\int_0^T e^{-Kr}
\mathcal{W}^2_2\big(\mathcal{L}(X^1_r,Y^1_r,Z^1_r),\mathcal{L}(X^2_r,Y^2_r,Z^2_r)\big)\,dr
\end{split}
\end{equation}
Using \eqref{27czerwca3}, arguing as in \eqref{24czerwca7}, and since $\frac{T\vee 1}{\beta}<1$,
\begin{equation}\label{27czerwca4}
\begin{split}
&\bigg(\frac{\beta-(T\vee 1)}{\beta}\bigg)
E\sup_{0\le t\le T}e^{-Kt}|X^1_t-X^2_t|^2\\
&\quad\le \frac{2(T\vee 1)}{\beta}
\Big(E\sup_{0\le t\le T}e^{-Kt}|Y^1_t-Y^2_t|^2
+E\int_0^T e^{-Kr}|Z^1_r-Z^2_r|^2\,dr\Big)
\end{split}
\end{equation}
Combining \eqref{27czerwca4} with \eqref{24czerwca9} yields
\begin{equation*}
\begin{split}
&E\sup_{0\le t\le T}e^{-Kt}|X^1_t-X^2_t|^2\\
&\quad\le
\frac{2(T\vee 1)\cdot 20\big(2l+(T\vee 1)+\frac{1}{20}\big)}{\beta-(T\vee 1)}
E\sup_{0\le t\le T}e^{-Kt}|x^1_t-x^2_t|^2,
\end{split}
\end{equation*}
where, by assumption, $\frac{2(T\vee 1)\cdot 20(2l+(T\vee 1)+\frac{1}{20})}{\beta-(T\vee 1)}<1$. Hence, $\phi\circ\psi$ is a contraction on $S^{2,K}_{\mathbb{F}}(0,T):=\{X\in\mathcal{S}_{\mathbb{F}}(0,T);\,E\sup_{0\le t\le T}e^{-Kt}|X_t|^2<\infty\}$ with the norm
\[
||X||_{S^{2,K}_{\mathbb{F}}(0,T)}:=\Big(E\sup_{0\le t\le T}e^{-Kt}|X_t|^2\Big)^{\frac{1}{2}}.
\]
Banach's fixed-point theorem yields a unique process $X\in\mathcal{S}^2_{\mathbb{F}}(0,T)$ such that $\psi(X)=(Y,Z)\in\mathcal{S}^2_{\mathbb{F}}(0,T)\otimes\mathcal{H}^2_{\mathbb{F}}(0,T)$ and $\phi(Y,Z)=X$. By the definition of $\psi$, there exists $R\in\mathcal{V}^2_{0,\mathbb{F}}(0,T)$ such that $(X,Y,Z,R)$ is a solution to MF-FBSDER$^T(\xi,h,B,\sigma,f,L,U)$.
\end{proof}

\section*{Acknowledgments}

This work was initiated during Maurycy Rzymowski's 2023 research visit to the Department of Mathematics at the University of Michigan. The visit was fully funded by Nicolaus Copernicus University in Toru\'n through its ``Mobility for Employees'' competition, conducted under the ``Excellence Initiative--Research University'' program.


\end{document}